\documentclass[11pt]{amsart}
\usepackage[margin=3cm]{geometry}
\usepackage{comment}
\usepackage{enumitem}
\usepackage[utf8]{inputenc}
\usepackage{float} 
\usepackage[usenames, dvipsnames]{color}
\usepackage{bm}
\usepackage{url}

\setlist[enumerate]{leftmargin=*, label=\alph*)} 
\setlist[itemize]{leftmargin=*} 

\usepackage{amssymb,amsthm,amsfonts,amsmath}

\newtheoremstyle{example-style}{5pt}{0pt}{}{}{\scshape}{:}{.5em}{}

\newtheorem{Thm}{Theorem}

\newtheorem{lem}{Lemma}
\newtheorem{cor}{Corollary}
\newtheorem{prop}{Proposition}
\newtheorem*{Problem*}{Problem}

\DeclareMathOperator{\Res}{Re}
\DeclareMathOperator{\Ims}{Im}

\newcommand*{\bfrac}[2]{\genfrac{}{}{0pt}{}{#1}{#2}}

\theoremstyle{definition}

\newtheorem*{Def*}{Definition}

\newtheorem*{remark*}{Remark}
\newtheorem{cond}{Condition}

\newcommand{\Je}{J_{\bm{e}}}
\newcommand{\Se}{\mathfrak{S}_{\bm{e}}}

\newcommand{\totient}{\phi_{{\theta}}}

\begin{document}
\title{Jordan totient quotients}
\author[P. Moree]{Pieter Moree}
\author[S. Saad Eddin]{Sumaia Saad Eddin}
\author[A. Sedunova]{Alisa Sedunova}
\author[Y. Suzuki]{Yuta Suzuki}

\address[P. Moree]{%
Max-Planck-Institut f\"ur Mathematik\\
Vivatsgasse 7\\D-53111 Bonn\\Germany.}
\email[P. Moree]{moree@mpim-bonn.mpg.de}
\address[A. Sedunova]{Max-Planck-Institut f\"ur Mathematik\\
Vivatsgasse 7\\D-53111 Bonn\\Germany.
 St.\,Petersburg State University\\ 14th Line 29B\\ Vasilyevsky Island\\ St.\,Petersburg\\ Russia.}
\email[A. Sedunova]{alisa.sedunova@phystech.edu}
\address[S. Saad Eddin]{%
Institute of Financial Mathematics and Applied Number Theory\\
JKU Linz\\Altenbergerstra{\ss}e 69\\4040 Linz\\Austria.}
\email{sumaia.saad\_eddin@jku.at}
\address[Y. Suzuki]{%
Graduate School of Mathematics\\
Nagoya University\\Furo-cho\\Chikusa-ku\\Nagoya\\Japan.}
\email{suzuyu1729@gmail.com}
\date{}

\maketitle

\begin{abstract}
The Jordan totient $J_k(n)$ can be
defined by $J_k(n)=n^k\prod_{p\mid n}(1-p^{-k})$.
In this paper, we study the average behavior
of fractions $P/Q$ of two products $P$ and $Q$ of Jordan totients,
which we call Jordan totient quotients.
To this end, we describe two general and 
ready-to-use methods that allow one to deal with a larger class of totient functions. The first one is elementary and the second one uses an 
advanced method due to Balakrishnan and P\'etermann. 
As an application, we determine
the average behavior of the Jordan totient quotient,
the $k^{th}$ normalized derivative
of the $n^{th}$ cyclotomic polynomial $\Phi_n(z)$ at $z=1$,
the second normalized derivative
of the $n^{th}$ cyclotomic polynomial $\Phi_n(z)$ at $z=-1$,
and the average order
of the Schwarzian derivative of $\Phi_n(z)$ at $z=1$.
\end{abstract}


\section{Introduction}
\subsection*{Jordan totient quotients}
Let $k\ge 1$ be an integer.
The {\em $k^{th}$ Jordan totient function} $J_k(n)$
counts the number of $k$-tuples chosen
from a complete residue system modulo $n$
such that the greatest common divisor of each tuple is coprime to $n$.
It is not difficult to show that
\begin{equation}
\label{heelflauw}
J_k(n)= n^{k} \prod_{p\mid n} \left(1-\frac{1}{p^k} \right),
\end{equation}
where $p$ here, and indeed in the whole paper, denotes a prime number.
The Jordan function first showed up in the work of Camille Jordan in 1870
in formulas for the order of finite matrix groups
(such as GL$(m,\mathbb Z/n\mathbb Z)$). 
For an introduction to Jordan totients see Section \ref{gejo}.

\begin{Def*}
Let $r\ge 1$ be an integer
and ${\bm e}=(e_1,\ldots,e_r)$ be a vector with integer entries.
Put $w=\sum_i ie_i$. The arithmetic function $J_{\bm e}$ of the form
\begin{equation}
\label{Jordan_rewritten}
J_{\bm e}(n)
=
\prod_{i=1}^{r}J_{i}^{e_i}(n) 
=
n^w \prod_{p\mid n}\prod_{i=1}^{r} \left(1-\frac{1}{p^{i}} \right)^{e_i},
\end{equation}
is said to be a \textit{Jordan totient quotient of \emph{weight} $w$}.
If $w=0$, then we say that it is a \textit{balanced Jordan totient quotient}, otherwise we
call it \emph{unbalanced}.
\end{Def*}

\begin{Def*}
For a Jordan totient $\Je$ quotient of weight $w$ we define $\Se$ as
\begin{equation}
\label{MV_Jordan:singular_e}
\Se
=
\prod_{p}\left(1+\frac{\Je(p)p^{-w}-1}{p}\right).
\end{equation}
\end{Def*}
The convergence of $\Se$ is ensured since
\[
1+\frac{J_{\bm{e}}(p)p^{-w}-1}{p}
=
1+O(p^{-2}).
\]
As $J_{\bm{e}}(p)p^{-w}>0>1-p,$ we have $\Se>0$.
This constant can be expanded as a product of partial zeta values,
see Moree and Niklasch~\cite{Mconstants,constants}.
As partial zeta values can be easily evaluated up to high precision
(say, with thousand decimals), this then allows one to do the same for $\Se$.


Note that if $J_{\bm e}$ is balanced,
then  $J_{\bm e}(n)$ depends only on the square-free kernel of $n$.
A famous (unbalanced) Jordan totient quotient is
\textit{the Dedekind $\Psi$-function} defined by
\[
\Psi(n)
=
n \prod_{p \mid n} \left(1 + \frac{1}{p}\right)
=
\frac{J_2(n)}{J_1(n)},
\]
which showed up in the work of Dedekind  on modular forms.

In this paper we study the average behavior of Jordan totient quotients.
In the remainder of the introduction we describe our main results, 
including an application to the study
of the average of the normalized derivative of cyclotomic polynomials.

Our first result gives an asymptotic formula
for the summatory function of any balanced Jordan totient quotient $\Je(n)$,
which implies that $\Je(n)$ is constant on average.

\begin{Thm}
\label{Thm:MV_Jordan}
Let $r \in \mathbb{N}$,
$\bm{e}=(e_1,\ldots, e_r) \in \mathbb{Z}^{r}$ be a vector of integers,
and $\Je$ be a Jordan totient quotient of weight $w=\sum_i i e_i=0$.
Then asymptotically 
\[
\sum_{n\le x}\Je(n)
=
\Se x
+
\sum_{r=1}^{|e_1|}C_{\bm{e},r}(\log x)^r
+
O_{\bm{e}}((\log x)^{2|e_1|/3}(\log\log x)^{4|e_1|/3}),
\]
where $\Se$ is given by \eqref{MV_Jordan:singular_e} and the $C_{\bm{e},r}$ are some constants%
\footnote{Work in progress by the fourth author \cite{Yuta} suggests
that the exponent $4|e_1|/3$ of $\log \log x$
in the error term can be decreased to $|e_1|/3$.}.
\end{Thm}

In case $\bm{e}=(0)$ is the zero vector,
then $J_{\bm{e}}(n)=1$ for every $n\ge 1$, 
$\Se=1$ and Theorem \ref{Thm:MV_Jordan} merely
states that $\sum_{n\le x}1=x+O(1)$.

We consider not only the balanced Jordan totient quotients,
but also a more general class of totient functions
(see Section \ref{Section:poznan} for the definitions).
This class is similar to the one earlier studied by Kaczorowski \cite{Kaczorowski}
in the context of inverse theorems for the Selberg class.
An analog of Theorem \ref{Thm:MV_Jordan} for non-zero weight
can be easily established on invoking Lemma \ref{Lem:weight_totient_BP}, 
and partial summation, 
but has such a long winding
formulation that we leave writing this down to the
interested reader.
\par The proof of Theorem~\ref{Thm:MV_Jordan} uses
the method of Balakrishnan and P\'etermann~\cite{balpet}, but before applying it (in Section \ref{Section:BP}),
we develop a simpler argument (see Section~\ref{Section:poznan}),
which actually applies to a wider class of totients.
This method allows us to get the main term of Theorem~\ref{Thm:MV_Jordan},
however only with a weaker error term. 
Theorems~\ref{Thm:MV_Jordan} 
and \ref{exconjJ} can be established 
also for non-zero weight by elementary means (see Proposition \ref{Prop:MV_Jordan_toy}).
\begin{Thm} 
\label{exconjJ}
Let $r \in \mathbb{N}$,
$\bm{e}=(e_1,\ldots, e_r) \in \mathbb{Z}^{r}$ be a vector of integers,
and $\Je$ be a Jordan totient quotient of weight $w=\sum_i i e_i=0$.
Then asymptotically
\[
\sum_{n \leq x} J_{\bm{e}}(n)=\Se x+O_{\bm{e}}((\log x)^{|e_1|}),
\]
where the constant $\Se$ is positive and given by \eqref{MV_Jordan:singular_e}.
\end{Thm}


It is an open problem to obtain a result
at least as strong as Theorem~\ref{Thm:MV_Jordan}
by more elementary methods than that used by Balakrishnan and P\'etermann.

\subsection*{Applications}
In Section \ref{apps} of the present paper,
we consider normalized higher derivatives of cyclotomic polynomials at $1$.
Our main result shows that they are constant on average.
We use the standard notation $\Phi_n$ and $B_n$
for the $n^{th}$ cyclotomic polynomial
and $n^{th}$ Bernoulli number, respectively (cf.~Section \ref{sec:prelim}).

\begin{Thm} 
\label{1overEuler}
Let $k\ge 1$.
There exist
a computable constant $\mathfrak{S}_k(\Phi)$ and constants $C_1,\ldots,C_r$
such that asymptotically 
\[
\sum_{1<n\le x}\frac{1}{\varphi(n)^k}\frac{\Phi_n^{(k)}(1)}{\Phi_n(1)}
=
\mathfrak{S}_k(\Phi)x
+
\sum_{r=1}^{k} C_r (\log x)^r
+O_k((\log x)^{2k/3}(\log \log x)^{4k/3}),
\]
where the constant $\mathfrak{S}_k(\Phi)$ is defined by
\begin{equation}
\label{devsingular}
\mathfrak{S}_k(\Phi)
=
(-1)^k
k!\sum\limits_{(\ast)}
\prod_{i=1}^{k}\frac{1}{\lambda_i!}
\left(\frac{B_i}{i!\cdot i}\right)^{\lambda_i}
\mathfrak{S}_{\bm{e}(\bm{\lambda})}
\end{equation}
with the summation $\sum\limits_{(\ast)}$
over all non-negative $\lambda_1, \ldots, \lambda_k$
with $\lambda_1+2\lambda_2+\ldots + k\lambda_k=k$,
and with the indices $\bm{e}(\bm{\lambda})=\bm{e}(\lambda_1,\ldots,\lambda_k)$
defined by
\begin{equation}
\label{deve}
\bm{e}(\bm{\lambda})
=
(e_i(\bm{\lambda}))_{i=1}^{\infty},\qquad
e_i(\bm{\lambda})
=
\left\{
\begin{array}{ll}
\lambda_1-k, & i=1,\\
\lambda_i, & 2\le i\le k,\\
0, & i>k.\\
\end{array}
\right.
\end{equation}
\end{Thm}

Note that the vectors $\bm{e}(\bm{\lambda})$
appearing as summands in \eqref{devsingular} are all balanced.
Neither can we predict the sign of $\mathfrak{S}_k(\Phi),$
nor can we exclude that $\mathfrak{S}_k(\Phi)=0$.

Although some part of the sum $\sum_{r=1}^{k} C_r (\log x)^r$
can be swamped by the error term, it turns out to be easier
to work with this full series rather than an appropriately truncated one.

In case $k=1$, we have by \eqref{phin1:deriv}
\begin{equation}
\label{1overEuler:trivial}
\sum_{1<n\le x}\frac{1}{\varphi(n)}
\frac{\Phi_n'(1)}{{\Phi_n(1)}}=\sum_{1<n\le x}\frac{1}{2}
=\frac{x}{2}+O(1),
\end{equation}
which is stronger that what Theorem \ref{1overEuler} yields. However, as our method of 
proof naturally includes the case $k=1,$ we have not excluded
it from our formulation of Theorem \ref{1overEuler}.

Theorem \ref{1overEuler} is a simple consequence
of Lemma \ref{lem:devphi} and Theorem~\ref{Thm:MV_Jordan}. 
We expect that an analogous result can be obtained
with $1$ replaced by any primitive root of unity of order $m$,
and that this would involve averages of generalized Jordan totients
(introduced in Bzd\c{e}ga et al.\,\cite{BHM}) of the form
\[
J_k(\chi;n)=\sum_{d\mid n} \mu(n/d)\chi(d) d^k,
\]
with $\chi$ a Dirichlet character of modulus $m.$ 
We will see such a result for $-1$ in case $k=2$
in the proof of Theorem \ref{Phi-1},
which is due to Herrera-Poyatos and the first author~\cite{HM}.
Finally, in Theorem~\ref{schwarziander},
we determine the average
of the Schwarzian derivative of $\Phi_n(z)$ evaluated at $z=1$. 


\section{The totient functions}\label{jordantot}
\label{gejo}
Let $k\ge 1$ be an integer
and $J_k(n)$ be the $k^{th}$ Jordan totient function.
This is one of many generalizations of Euler's totient
function (the case $k=1$), see Sivaramakrishnan~\cite{Siva-1}.
It is easy to see, cf.~\cite[p. 91]{Siva}, that
\[
n^k=\sum_{d\mid n}J_k(d),
\]
which, by M\"obius inversion, yields
\begin{equation}
\label{geenidee}
J_k(n)
=
\sum_{d\mid n}\mu(d)\left(\frac{n}{d}\right)^k.
\end{equation}
Thus $J_k$ is a Dirichlet convolution of two multiplicative functions
and hence is itself multiplicative. By the Euler product formula,
it then follows from \eqref{geenidee} that \eqref{heelflauw} holds true.

Given a Jordan totient quotient of weight $w=\sum_{i} i e_i$
as in \eqref{Jordan_rewritten}, we normalize it by dividing by $n^{w}$,
and this results in
\begin{equation}
\label{Jordan_normalized}
\frac{\Je(n)}{n^{w}}
=\prod_{p\mid n}\prod_{i=1}^{r}\left(1-\frac{1}{p^i}\right)^{e_r}=\prod_{p\mid n}\left(1-\frac{e_1}{p}+O\left(\frac{1}{p^2}\right)\right).
\end{equation}
Although our aim is to study this particular function,
our methods easily allow a more general class of totients to be dealt with.
\begin{Def*}[General totient]
Let $\theta_n$ be a complex valued multiplicative function
supported on square-free numbers.
Define \textit{the ${\theta}$-totient} $\totient(n)$ by
\[
\phi_{\theta}(n)
=
\prod_{p\mid n} (1+\theta_p) = \sum_{d\mid n}\theta_d.
\]
\end{Def*}

It is easy to see that any arithmetic function
$f$ that only depends on the square-free kernel of $n$ for every $n\ge 1$,
is of the form $\totient$ for some ${\theta}$.

We next list conditions on $\theta$ 
occurring in this paper. In the formulation of
several results we will specify which 
of these
conditions are being used.

\begin{cond}\label{Theta1}
There exist non-negative constants $\sigma, \kappa, A$ with $0\le\sigma<1$
such that for any $x\ge2$ we  have
\begin{equation*}
\sum_{p\le x}\frac{|\theta_p|}{p^{\sigma}}\le\kappa\log\log x+A.
\end{equation*}
\end{cond}

\begin{cond} \label{Theta2}
There exist $0<\lambda<1/2$ and $\alpha\in\mathbb{R}$ with $|\alpha|\ge1$
such that for all primes $p$ we have
\[
\theta_p
=
\alpha/{p} + r_p, \quad r_p= O(p^{-1-\lambda}).
\]
\end{cond}

\begin{cond} \label{Theta3}
With respect to $p$ the function $p\theta_p$ is ultimately monotonic%
\footnote{Condition \ref{Theta3} can be removed at the expense
of more technicalities, see \cite{Yuta}.}.
\end{cond}

Note that if Condition \ref{Theta2} is satisfied,
then so is Condition \ref{Theta1} with $\sigma=0$ and $\kappa=|\alpha|$.
We point out that in order
to prove Theorem \ref{exconjJ} only Condition~\ref{Theta1} is needed,
whereas to prove Theorem \ref{Thm:MV_Jordan}
we shall impose the stronger Condition \ref{Theta2}.
Notice that if $\theta$ is defined by $\Je(n)/n^w = \phi_\theta(n)$, 
then Condition \ref{Theta2} is satisfied
with $\alpha=-e_1$ and $\lambda=1$, cf.~\eqref{Jordan_normalized}.

\section{Mean values of general totients via an elementary method}\label{Section:poznan}
In this section, we give a simple method to obtain asymptotic formulas
for the mean value of multiplicative functions of a certain type.
The ideas and techniques are not new,
but our aim is to provide a quick way to translate the definition
of multiplicative functions to the asymptotic formula of its mean value.
As we have seen, our $\theta$-totient is modeled
on the normalized Jordan totient quotient \eqref{Jordan_normalized}.
Thus we need to introduce a weight factor $n^{\beta}$.

\begin{lem}
\label{Lem:mbeta}
Let $\beta$ be an arbitrary real number.
For $x\ge1$ we have
\[
\sum_{n \le x}n^{\beta}
=
M_{\beta}(x)+C_{0}(\beta)+O_{\beta}(x^{\beta}),
\]
where $C_0(\beta)$ is a constant depending only on $\beta$,
$M_{-1}(x)=\log x$ and
\[
M_{\beta}(x)
=
\frac{x^{\beta+1}}{\beta+1}, \text{~if~}\beta\neq-1.
\]
\end{lem}
\begin{proof}
Follows from parts (a), (b), and (d) of \cite[Theorem~3.2]{Apostolbook}.
\end{proof}

\begin{lem}
\label{Lem:MV_totient}
Let $\phi_{{\theta}}$ be a $\theta$-totient
and $\beta$ be an arbitrary real number.
Assume that $\theta$ satisfies Condition \ref{Theta1}. 
We then have
\[
\sum_{n\le x}n^{\beta}\totient(n)
=\mathfrak{S}_{\theta}M_{\beta}(x)
+C(\theta,\beta)
+O_{\sigma,\kappa, A, \beta}(x^{\sigma+\beta}(\log x)^{\kappa}),
\]
where $\mathfrak{S}_{\theta}$ is given by the absolutely convergent product 
\begin{equation}\label{Sigma(theta)}
\mathfrak{S}_{\theta}=\prod_{p}\left(1+\frac{\theta_p}{p}\right),
\end{equation}
and $C(\theta,\beta)$ is a constant
depending only on $\theta$ and $\beta$.
\end{lem}
\begin{proof}
By the definition of $\theta$-totient, we have
\begin{equation*}
\sum_{n\le x}n^{\beta}\totient(n)=
\sum_{n\le x}n^{\beta}\sum_{d\mid n} \theta_d=
\sum_{d\le x}d^{\beta}\theta_d\sum_{m\le x/d}m^{\beta}.
\end{equation*}
Thus, by Lemma~\ref{Lem:mbeta}, we have
\begin{equation*}
\sum_{n\le x}n^{\beta}\totient(n)
=
\frac{x^{\beta+1}}{\beta+1}
\sum_{d\le x}\frac{\theta_d}{d}
+
C_{0}(\beta)\sum_{d\le x}d^{\beta}\theta_d
+
O_{\beta}\Big(x^{\beta}\sum_{d\le x}|\theta_d|\Big)
\end{equation*}
if $\beta\neq-1$, and
\begin{equation*}
\sum_{n\le x}n^{\beta}\totient(n)
=
\sum_{d\le x}\frac{\theta_d}{d}\log\frac{x}{d}
+
C_{0}(\beta)\sum_{d\le x}\frac{\theta_d}{d}
+
O\Big(\frac{1}{x}\sum_{d\le x}|\theta_d|\Big)
\end{equation*}
if $\beta=-1$.
Using Condition \ref{Theta1}, we find that
\begin{equation*}
\sum_{d\le x}\frac{|\theta_d|}{d^{\sigma}}\le
\prod_{p\le x}\Big(1+\frac{|\theta_p|}{p^{\sigma}}\Big)\le
\exp\Big(\sum_{p\le x}\frac{|\theta_p|}{p^{\sigma}}\Big)
\ll_{\sigma, \kappa, A}
(\log x)^\kappa.
\end{equation*}
This implies that for $\beta\ge-\sigma$
\begin{equation*}
\sum_{d\le x}d^{\beta}|\theta_d|
\le
x^{\sigma+\beta}\sum_{d\le x}\frac{|\theta_d|}{d^{\sigma}}
\ll_{\sigma,\kappa,A}
x^{\sigma+\beta}(\log x)^{\kappa},
\end{equation*}
and that for $\beta<-\sigma$
\begin{equation*}
\begin{aligned}
\sum_{d>x}d^{\beta}|\theta_d|
&\ll_{\sigma,\beta}
\sum_{d>x}\frac{|\theta_d|}{d^{\sigma}}
\int_{d}^{\infty}u^{\sigma+\beta-1}du
\ll_{\sigma,\beta}
\int_{x}^{\infty}
\Big(\sum_{d\le u}\frac{|\theta_d|}{d^{\sigma}}\Big)
u^{\sigma+\beta-1}du\\
&\ll_{\sigma,\kappa,A,\beta}x^{\sigma+\beta}(\log x)^{\kappa}.
\end{aligned}
\end{equation*}
Hence, in particular,
\[
\sum_{d\le x}|\theta_d|
\ll
x^{\sigma}(\log x)^{\kappa},
\]
and
\[
\sum_{d\le x}\frac{\theta_d}{d}
=
\mathfrak{S}_{\theta}
+r(x),\quad
r(x)\ll_{\sigma,\kappa,A}x^{\sigma-1}(\log x)^{\kappa}.
\]
By combining the above,
we obtain the assertion in case $\beta\neq-1$.

For the case $\beta=-1$, we have to evaluate the main term.
We obtain
\begin{align*}
\sum_{d\le x}\frac{\theta_d}{d}\log\frac{x}{d}
&=
\sum_{d\le x}\frac{\theta_d}{d}\int_{d}^{x}\frac{du}{u}
=
\int_{1}^{x}\Big(\sum_{d\le u}\frac{\theta_d}{d}\Big)
\frac{du}{u}\\
&=
\mathfrak{S}_{\theta}\log x
+
\int_{1}^{x}\frac{r(u)}{u}du
=
\mathfrak{S}_{\theta}\log x
+
\int_{1}^{\infty}\frac{r(u)}{u}du
-
\int_{x}^{\infty}\frac{r(u)}{u}du.
\end{align*}
The last integral can be estimated as
\[
\int_{x}^{\infty}\frac{r(u)}{u}du
\ll
\int_{x}^{\infty}u^{\sigma-2}(\log u)^{\kappa}du
\ll
x^{\sigma-1}(\log x)^{\kappa},
\]
since $\sigma<1$ by assumption.
This completes the proof when $\beta=-1$.
\end{proof}

As a special case we obtain the following result
involving the Jordan totient quotient.
\begin{prop}
\label{Prop:MV_Jordan_toy}
Let $\bm{e}=(e_1,\ldots, e_r) \in \mathbb{Z}^{r}$ be a vector of integers
and $\Je(n)$ be the associated Jordan totient quotient of weight $w=\sum_i ie_i$.
For any real number $\beta$ we have
\[
\sum_{n\le x}\Je(n)n^{\beta}
=
\Se M_{\beta+w}(x)+
C(\bm{e},\beta)+O_{\bm{e},\beta}(x^{\beta+w}(\log x)^{|e_1|}),
\]
where $\Se$ is given by \eqref{MV_Jordan:singular_e}  and $C(\bm{e},\beta)$ is a constant
depending only on $\bm{e}$ and $\beta$. 
\end{prop}
\begin{proof}
We can regard $\Je(n)n^{-w}$ as a general totient $\totient(n)$
with components 
\[
\theta_p=-e_1/p+O(p^{-2}).
\]
This gives
\[
\sum_{p\le x}|\theta_p|
=
\sum_{p\le x}\frac{|e_1|}{p}+O(1)
=
|e_1|\log\log x+O(1),
\]
i.e.~$\theta$ 
satisfies Condition \ref{Theta1}
with $\sigma=0$, $\kappa=|e_1|$.
Note that 
\[
\theta_p
=
J_{\bm{e}}(p)p^{-w}-1,
\]
and so comparison of 
\eqref{Sigma(theta)} and
\eqref{MV_Jordan:singular_e} 
yields $\mathfrak{S}_\theta=\Se$.
Under the above setting, we can write 
\[
\sum_{n\le x}\Je(n)n^{\beta}
=
\sum_{n\le x}n^{\beta+w}\Big(\frac{\Je(n)}{n^{w}}\Big)
=
\sum_{n\le x}n^{\beta+w}\phi_{\theta}(n),
\]
and the proposition follows by Lemma \ref{Lem:MV_totient}. 
\end{proof}
\begin{cor}
\label{cor:philog}
For $k \geq 1$ we have
\begin{equation*}
\sum_{n\leq x}\frac{n^{k-1}}{\varphi(n)^k}=\mathfrak{S}_{(-k)}\log x+C_k+O_k\Big(\frac{(\log x)^k}{x}\Big),
\end{equation*}
where $\mathfrak{S}_{(-k)}$ is given by \eqref{MV_Jordan:singular_e} and  $C_k$
is a constant depending on $k$.
\end{cor}
\begin{proof}
Apply Proposition \ref{Prop:MV_Jordan_toy} with $\Je(n)=\varphi(n)^{-k}$, $w=-k$ and $\beta=k-1$.
\end{proof}

\section{Mean values of general totients by Balakrishnan-P\'{e}termann}
\label{Section:BP}
\label{Section:weight_BP}
In this section we use the method of Balakrishnan and P\'etermann~\cite{balpet}
in order to prove Theorem~\ref{Thm:MV_Jordan}. 
The method builds on Propositions~\ref{Thm:BP1} and \ref{Thm:BP2} below and
yields an asymptotic formula for the mean value of $\theta$-totients,
provided some condition stronger than Condition \ref{Theta1} is satisfied.

\begin{prop}[{Balakrishan and P\'{e}termann \cite[Theorem~1]{balpet}}]
\label{Thm:BP1}
Let
\begin{equation*}
f(s)=\sum_{n=1}^{\infty}\frac{b_n}{n^s}
\end{equation*}
be a Dirichlet series that converges absolutely for $\sigma>1-\lambda$,
with $\lambda$ a positive real number.
Define two arithmetic functions $a_n$ and $v_n$ by
\[
\sum_{n=1}^{\infty}\frac{a_n}{n^s}
=
\zeta(s)\zeta(s+1)^{\alpha}f(s+1), \quad\quad
\sum_{n=1}^{\infty}\frac{v_n}{n^s}=\zeta(s)^{\alpha}f(s),
\]
where $\sigma>1$, $\alpha$ is an arbitrary real number
and the branch of $\zeta(s+1)^{\alpha}$ is taken
by the one for which $\arg\zeta(s+1)$ equals zero on the positive real line.
Then we have
\[
\sum_{n\le x}a_n
=
\zeta(2)^{\alpha}f(2)x
+
\sum_{r=0}^{[\alpha]}A_r(\log x)^{\alpha-r}
+R(x)+o(1)
\]
as $x\to\infty$, where the coefficients $A_r$ are computable
from the Laurent expansion of $\zeta(s)^{\alpha}f(s)$ at $s=1$,
the remainder term $R(x)$ is given by
\[
R(x)=\sum_{n\le y}\frac{v_n}{n}\psi\left(\frac{x}{n}\right),
\]
with $y=x\exp(-(\log x)^{1/6})$, and $\psi(x)=\{x\}-1/2$. 
The implicit constant in the error term might depend on all the input data.
\end{prop}

\begin{remark*}
There are several results in \cite{balpet} that
depend on the specific zero-free region
for the Riemann zeta function being used.
Throughout this paper, we will use the zero-free region
\[
\Res s  \geq 1-\frac{1}{(\log t)^{4/5}}
\quad \text{ and }\quad
\Ims s\ge t_0,
\]
where $t_0$ is some large constant, see \cite[Eq.\,(6.15.1)]{Titchmarsh}.
This zero-free region enables us to take $b=1/6$ in \cite{balpet}, 
leading to $y=x\exp(-(\log x)^{1/6})$ in Proposition \ref{Thm:BP1}, 
see \cite[Subsection 1.4, Lemmas~3 and 5]{balpet}.
\end{remark*}

\begin{lem}[{Balakrishnan and P\'{e}termann
\cite[Lemma~3]{balpet}}]
\label{Lem:BP_Lemma3}
In the notation of Proposition \ref{Thm:BP1} we have 
\[
\sum_{n\le x}\frac{v_n}{n}
=
\sum_{0\le r\le (\log x)^{1/6}}V_r(\log x)^{\alpha-r}
+O(\exp(-(\log x)^{1/6})),
\]
with
$|V_r|\le(cr)^{r}$ for every $r\ge 1$ and
$c \ge 1$ a constant possibly depending on $v$. 
\end{lem}

Now we prove Theorem~\ref{Thm:MV_Jordan}.
As already mentioned, we need to assume that $\theta$ satisfies
a stronger condition than Condition \ref{Theta1}.
In this section, we use Conditions \ref{Theta2} and \ref{Theta3}, 
and hence all implicit constants in this section will
depend on the constants $\alpha,\lambda$ and the implicit constant
appearing in Condition \ref{Theta2}.

\begin{lem}
\label{Lem:check_BP1}
Let $\phi_\theta$ be a $\theta$-totient with $\theta$ satisfying Condition~\ref{Theta2}.
Consider the formal Dirichlet series
\[
f(s+1)
=
\sum_{n=1}^{\infty}
\frac{b_n}{n^{s+1}}
=
\zeta(s)^{-1}\zeta(s+1)^{-\alpha}\sum_{n=1}^{\infty}\frac{\totient(n)}{n^s},
\]
where $\alpha$ is the same one as in Condition~\ref{Theta2}.
Then $f(s)$ converges absolutely for $\Res s >1-\lambda$.
\end{lem}
\begin{proof}
By the definition of $\phi_\theta$ we have 
\begin{equation}
\label{check_BP1:f_rewritten}
\sum_{n=1}^{\infty}\frac{\theta_n}{n^s}
=
\zeta(s+1)^{\alpha}f(s+1).
\end{equation}
If we consider the Dirichlet series given by
\[
\zeta(s)^{-\alpha}=\sum_{n=1}^{\infty}\frac{\tau_{-\alpha}(n)}{n^s},
\]
then, using \eqref{check_BP1:f_rewritten} for the coefficients of $f(s),$ we obtain
\begin{equation}
\label{check_BP1:f_termwise}
b_n=\sum_{dm=n}\tau_{-\alpha}(d)\theta_m m.
\end{equation}
Using the Euler product expansion
and the generalized binomial formula, we see that
\begin{equation}
\label{check_BP1:Euler_product}
\zeta(s)^{-\alpha}
=
\prod_{p}\left(1-\frac{1}{p^s}\right)^{\alpha}
=
\prod_{p}\left(1+
\sum_{\nu=1}^{\infty}(-1)^{\nu}\binom{\alpha}{\nu}\frac{1}{p^{\nu s}}\right)=:\prod_p\left(1+H_{\alpha}(p^{-s})\right),
\end{equation}
where
\[
\binom{\alpha}{\nu}
=
\frac{1}{\nu!}\prod_{\ell=0}^{\nu-1}(\alpha-\ell),
\]
is a {\em generalized binomial coefficient}.
Since
\[|H_{\alpha}(p^{-s})|
=\frac{|\alpha|}{p^{\sigma}}
+O_{\alpha}\left(\frac{1}{p^{2\sigma}}\right),
\]
the Euler product \eqref{check_BP1:Euler_product}
is absolutely convergent for $\sigma = \Res s>1$ and
\begin{equation*}
\tau_{-\alpha}(p^{\nu})=(-1)^{\nu}\binom{\alpha}{\nu}.
\end{equation*}
Note that
\begin{align*}
|\tau_{-\alpha}(p^{\nu})| \le \left|\binom{\alpha}{\nu}\right|
\le
\frac{1}{\nu!}\prod_{\ell=1}^{\nu}
(|\alpha|+\ell-1)
&\le
\prod_{\ell=1}^{\nu}\left(1+\frac{|\alpha|}{\ell}\right)\\
&\le
\exp\left(\sum_{\ell=1}^{\nu}\frac{|\alpha|}{\ell}\right)
\ll
\nu^{|\alpha|}
\ll_{\varepsilon}
p^{\nu\varepsilon}
\end{align*}
for every $\varepsilon>0$.
Substituting $n=p,$ 
respectively $n=p^{\nu}$ into \eqref{check_BP1:f_termwise}
and using Condition \ref{Theta2}, we find that
\[
b_p=\tau_{-\alpha}(p)+p\theta_p=-\alpha+\alpha+pr_p=O(p^{-\lambda})
\]
and $b_{p^\nu}\ll_\varepsilon p^{\nu\varepsilon}$
for $\nu \geq 2$ and every $\varepsilon>0$.
As
\[
\sum_{n=1}^{\infty}\frac{|b_n|}{n^{\sigma}}
=
\prod_{p}\left(1+\frac{|b_p|}{p^{\sigma}}
+\sum_{\nu=2}^{\infty}\frac{|b_{p^{\nu}}|}{p^{\nu\sigma}}\right)\]
is bounded when both $\sigma+\lambda>1$ and $2\sigma>1$,
the result follows since $\lambda<1/2$.
\end{proof}

\begin{lem}
\label{Lem:MV_BP}
Let $\phi_\theta$ be a $\theta$-totient with $\theta$ satisfying Condition~\ref{Theta2}.
Then we have
\[
\sum_{n\le x}\totient(n)
=
\mathfrak{S}_{\theta}x
+\sum_{r=0}^{[\alpha]}C_r(\theta)(\log x)^{\alpha-r}
+R(x)
+o(1),
\]
where 
$\mathfrak{S}_{\theta}$ is given by \eqref{Sigma(theta)}, $R(x)$ by
\[
R(x)=\sum_{n\le y}
\theta_n\psi\left(\frac{x}{n}\right),
\]
and $y=x\exp(-(\log x)^{1/6})$.
\end{lem}
\begin{proof}
With the choice $a_n =\phi_\theta(n)$,
we are in the scope of Proposition~\ref{Thm:BP1} by Lemma~\ref{Lem:check_BP1}, 
and on applying it and noting that $v_n=n\theta_n$, we complete the proof.
\end{proof}

We next estimate the error term in Proposition~\ref{Thm:BP1}.
For this purpose, we need Theorem~1 of P\'etermann~\cite{Petermann},
which we state below%
\footnote{Note that \cite[Theorem~2]{balpet}
contains an error. See the errata of \cite{balpet}
and \cite{Petermann}.}.
Note that the parameter $\alpha$ in \cite{Petermann} corresponds to $|\alpha|-1$
in Proposition \ref{Thm:BP1}.
In order to avoid possible confusion caused by this clash of notation,
we replace $\alpha$ in \cite{Petermann} by $\alpha_1$.

\begin{prop}[P\'etermann {\cite[Theorem~1]{Petermann}}]
\label{Thm:BP2}
Let $v_n$ be a real-valued multiplicative function. Assume that there exist
real numbers $\alpha_1,\beta\ge0$,
and a sequence of real numbers $\{V_r\}_{r=0}^{\infty}$,
such that for every integer $B > 0$ and real number $x \ge 4$, we have
\begin{equation}\label{H1}\tag{\textbf{h1}}
\sum_{n\le x}|v_n|=
x\sum_{r=0}^{B+[\alpha_1]}V_r(\log x)^{\alpha_1-r}+
O_{B}(x(\log x)^{-B}),
\end{equation}
\begin{equation}
\label{H2}
\tag{\textbf{h2}}
\sum_{n\le x}|v_n|^2\ll x(\log x)^{\beta},
\end{equation}
\begin{equation}
\label{H3}
\tag{\textbf{h3}}
\begin{aligned}
&v_p \text{ is ultimately monotonic with respect to }p,\\
&v_{p^{\nu}}\text{ is bounded as
$p^{\nu}$ runs over the prime powers}.
\end{aligned}
\end{equation}
Then, for $x\ge4$, we have
\[
\sum_{n\le y}\frac{v_n}{n}\psi\left(\frac{x}{n}\right)
\ll
(\log x)^{2(\alpha_1+1)/3}(\log\log x)^{4(\alpha_1+1)/3},
\]
where $y=x\exp(-(\log x)^{\frac{1}{6}})$
and the implicit constant depends on the constants 
in Conditions
{\upshape(\ref{H1})}, {\upshape(\ref{H2})} and {\upshape(\ref{H3})}.
\end{prop}

We now apply Proposition \ref{Thm:BP2} to our setting.
For this purpose, we need Lemma~\ref{Lem:BP_Lemma3}
(which can, in principle, also be proven via the Selberg--Delange method).

\begin{lem}\label{Lem:weight_totient_BP}
Let $\phi_\theta$ be a $\theta$-totient. Assume that $\theta$
satisfies Conditions {\upshape\ref{Theta2}} and {\upshape\ref{Theta3}}.
Then we have
\[
\sum_{n\le x}\totient(n)=
\mathfrak{S}_{\theta}x
+\sum_{r=0}^{[\alpha]}C_r(\theta)(\log x)^{\alpha-r}
+O((\log x)^{2|\alpha|/3}(\log\log x)^{4|\alpha|/3}),
\]
where $\mathfrak{S}_{\theta}$ is given by \eqref{Sigma(theta)}.
Furthermore, for $\beta$ real we have
\[
\sum_{n\le x}n^\beta\totient(n)=
\mathfrak{S}_{\theta}M_\beta(x)+
C(\theta,\beta)+\sum_{r=0}^{[\alpha]}C_r(\theta,\beta)x^{\beta}(\log x)^{\alpha-r}+E(x;\beta),
\]
where $M_{\beta}(x)$ is defined in Lemma \ref{Lem:mbeta},
and
\[
E(x;\beta)\ll x^{\beta}(\log x)^{2|\alpha|/3}(\log\log x)^{4|\alpha|/3}.
\]
\end{lem}
\begin{proof}
By Lemma \ref{Lem:MV_BP},
it is sufficient to show that
$$R(x)=O((\log x)^{2|\alpha|/3}(\log\log x)^{4|\alpha|/3}),$$
which we do via Proposition \ref{Thm:BP2}.
Hence, we need to check that
Conditions \eqref{H1}, \eqref{H2} and \eqref{H3} are all satisfied.
Since $\theta_n$ satisfies Condition \ref{Theta2},
$|\theta_n|$ also satisfies Condition \ref{Theta2},
but with $|\alpha|$ instead of $\alpha$.
Thus, we can apply Lemma~\ref{Lem:check_BP1} with $|\theta_n|$ instead of $\theta_n$.
Then, as $v_n=n\theta_n$,
we can replace $v_n$ in Proposition~\ref{Thm:BP1} by $|v_n|$.

We start with Condition \eqref{H1}.
We apply Lemma~\ref{Lem:BP_Lemma3} and obtain
\[
\sum_{n\le x}\frac{|v_n|}{n}
=
\sum_{0\le r\le (\log x)^{1/6}}V_r(\log x)^{|\alpha|-r}
+O(\exp(-(\log x)^{1/6})),
\]
where the $V_r$ are some constants satisfying $|V_r|\le(cr)^r$
with some $c\ge1$. Let $B>0$ be an integer that is kept fixed.
Then it is easy to see that for $x$ larger than
some constant depending on $B$ and 
$\alpha,$
the consecutive terms of the sequence
\[
V_r(\log x)^{|\alpha|-r}
\quad\text{with}\quad
B+|\alpha|<r\le (\log x)^{1/6}
\]
have ratio $\le1/2,$ and so
their sum is bounded by the first 
term, as
\begin{align*}
\sum_{B+|\alpha|<r\le (\log x)^{1/6}}
V_r(\log x)^{|\alpha|-r}
&\ll
V_{B+[|\alpha|]+1}(\log x)^{|\alpha|-(B+[|\alpha|]+1)}\\
&\ll_{B}
(\log x)^{-B}.
\end{align*}
This enables us to truncate the sum over $r$ to obtain
\[
S(x):=
\sum_{n\le x}\frac{|v_n|}{n}
=
\sum_{0\le r\le B+|\alpha|}
V_r(\log x)^{|\alpha|-r}
+R_{S}(x),\quad
R_{S}(x)\ll_{B}(\log x)^{-B}.
\]
By partial summation, we have
\begin{align*}
\sum_{n\le x}|v_n|
&=
\int_{2}^{x}u\, dS(u)+O(1)\\
&=
\sum_{0\le r\le B+|\alpha|}
(|\alpha|-r)V_r\int_{2}^{x}(\log u)^{|\alpha|-r-1}du
+\int_{2}^{x}u\, dR_{S}(u)+O(1).
\end{align*}
The main terms can be evaluated using integration by parts as
\[
\int_{2}^{x}(\log u)^{|\alpha|-r-1}\, du
=
\sum_{0\le m\le B+|\alpha|-r-1}
C_{m}x(\log x)^{|\alpha|-r-1-m}
+
O_{B}(x(\log x)^{-B}),
\]
with some constants $C_m$ depends on $\alpha$ and $r$.
The error term satisfies
\[
\int_{2}^{x}u\, dR_{S}(u)
\ll_{B}
x(\log x)^{-B}
+
\int_{2}^{x}(\log u)^{-B}du
\ll_{B}
x(\log x)^{-B}.
\]
By combining the above estimates, we arrive at
\[
\sum_{n\le x}|v_n|
=
x\sum_{r=0}^{B+[|\alpha|-1]}\tilde{V}_{r}(\log x)^{|\alpha|-1-r}
+
O(x(\log x)^{-B}),
\]
where the $\tilde{V}_{r}$ are constants.
By Condition~\ref{Theta2}, we have $|\alpha|\ge1$.
Hence, Condition \eqref{H1} of Proposition \ref{Thm:BP2} is satisfied with $\alpha_1=|\alpha|-1\ge0$.

As to Condition \eqref{H2}, we start  with the string of estimates
\begin{align}
\label{eq:blah1}
\sum_{n\le x}|v_n|^2
=
\sum_{n\le x}n^2|\theta_n|^2
\le
x\sum_{n\le x}n|\theta_n|^2
&\le
x\prod_{p\le x}\left(1+p|\theta_p|^2\right) \nonumber\\
&\le
x\exp\Big(\sum_{p\le x}p|\theta_p|^2\Big).
\end{align}
Now Condition \ref{Theta2} implies that
\begin{equation}
\label{eq:blah2}
\sum_{p\le x}p|\theta_p|^2
\ll
\sum_{p\le x}\frac{1}{p}
\ll
\log\log x.
\end{equation}
By combining 
\eqref{eq:blah1} and \eqref{eq:blah2},
we see that Condition \eqref{H2} is satisfied as well.

The remaining Condition \eqref{H3} follows immediately
from our setting and Condition~\ref{Theta3}.

Thus Conditions 
\eqref{H1}, \eqref{H2} and 
\eqref{H3} are satisfied and we 
get the claimed upper bound for 
$R(x),$ which on insertion in
Lemma~\ref{Lem:MV_BP} yields the first assertion of the lemma.
The second claim now follows by  partial summation.
\end{proof}

\begin{proof}[Proof of Theorem~1]
Consider the $\theta$-totient  $\phi_{\theta}(n)=J_{\bm{e}}(n)n^{-w}$.
Note that $\theta$ satisfies Condition~\ref{Theta2}
with $\alpha=-e_1$ and $\lambda=1/4$ 
and, moreover, satisfies Condition~\ref{Theta3}.
Thus, in case 
$e_1\neq0,$ Theorem~\ref{Thm:MV_Jordan} follows immediately from Lemma~\ref{Lem:weight_totient_BP}.
The case $e_1=0$ is just a corollary of Theorem~\ref{exconjJ}.
\end{proof}

\section{Applications}\label{apps}
\begin{Def*}
Let $f(X)\in \mathbb Z[X]$ be a polynomial
and let $\deg f$ denote its degree with respect to $X$.
For any complex number $z$ such that $f(z)\ne 0$, we define
\begin{equation*}
F^{(k)}(z)=\frac{1}{(\deg f)^k}\frac{f^{(k)}(z)}{f(z)}
\end{equation*}
as \emph{the normalized $k^{th}$ derivative of $f$ at $z$}. 
\end{Def*}

In case $f(X)\in \mathbb Z_{\ge 0}[X]$, $z\ge1$ is real, and $f(z)\ne 0$,
it is easy to show that $F^{(k)}(z) \le 1$.
This observation leads to the following problem.

\begin{Problem*}
Let $z$ be given.
Let $\mathcal F$ be an infinite family of polynomials $f$ with $f(z)\ne 0$.
Study the average behavior and value distribution
of $F^{(k)}(z)$ in the family $\mathcal F$.
\end{Problem*}

Here we consider the family $\mathcal F=\{\Phi_n:n\ge 2\}$, 
where $\Phi_n$ denotes the $n^{th}$ cyclotomic polynomial.
It can be defined by
\[
\Phi_n(X)
=
\prod_{\bfrac{1\le j\le n}{(j,n)=1}}(X-\zeta_n^j)
=
\sum_{k=0}^{\varphi(n)}a_n(k)X^k,
\]
with $\zeta_n$ any primitive $n^{th}$ root of unity. Note
that $\Phi_n(1)\ne 0$ for $n>1$ and that $\Phi_n(-1)\ne 0$ for $n>2$.
Theorem \ref{1overEuler} shows that 
$$\frac{1}{\varphi(n)^k}\frac{\Phi_n^{(k)}(1)}{\Phi_n(1)},$$
the $k^{th}$ normalized derivative 
of $\Phi_n$ at $1$,
is constant on averaging over $n$.

\subsection{The $k^{th}$ derivative of $\Phi_n$ at $1$}
\label{sec:prelim}
In this section we first recall some known results on 
$\Phi_n^{(k)}$. For a survey (and some new 
results) see
Herrera-Poyatos and Moree \cite{HM}.

The \emph{Bernoulli numbers} $B_n$ can be recursively
defined by
\begin{equation*} 
B_n= - \sum_{k=0}^{n-1}\binom{n}{k} \frac{B_k}{n-k+1},
\end{equation*}
with $B_0=1$.
The
coefficients $c(k,j)$ of the polynomial
$$X(X-1) \cdots (X-k+1)=\sum_{j=0}^{k}c(k,j)X^j$$
are called \emph{the signed Stirling numbers of the first kind}.

\begin{lem}[{Lehmer~\cite[Theorems~2 and 3]{Lehmer}}]
\label{lem:lehmer}
For $n>1$ and $k\ge1$, we have
\begin{equation}
\label{lehmer:newton}
\frac{\Phi_n^{(k)}(1)}{\Phi_n(1)}
=
k!\sum\limits_{(\ast)}
\prod_{i=1}^{k}\frac{(-s_i(n))^{\lambda_i}}{\lambda_i!i^{\lambda_i}}
\end{equation}
where the summation $\sum\limits_{(\ast)}$ is as in Theorem~\ref{1overEuler}
and
\begin{equation}
\label{lehmer:power}
s_i(n)
:=
-\frac{1}{(i-1)!}
\sum_{h=1}^{i}(-1)^{h}\frac{B_{h}}{h}c(i,h)J_{h}(n).
\end{equation}
\end{lem}

\begin{remark*}
Theorem~2 of \cite{Lehmer}
gives the formula
\[
s_i(n)
=
\frac{(-1)^i}{2}\varphi(n)
-
\frac{1}{(i-1)!}
\sum_{h=1}^{[i/2]}\frac{B_{2h}}{2h}c(i,2h)J_{2h}(n).
\]
This looks different from \eqref{lehmer:power}, but is actually
seen to be the same on noting that 
$B_1=-1/2$, $c(i,1)=(-1)^{i-1}(i-1)!,$ $J_1(n)=\varphi(n)$ and
$B_h=0$ for odd $h>1.$
\end{remark*}

In particular, using Lemma \ref{lem:lehmer} with $k=1,2$ for $n>1$ yields
\begin{equation}\label{phin1:deriv}
\frac{\Phi_n'(1)}{{\Phi_n(1)}}=\frac{\varphi(n)}{2},
\end{equation}
and 
\[
\frac{\Phi_n''(1)}{\Phi_n(1)}
=
\frac{\varphi(n)}{4}\left(\varphi(n) + \frac{\Psi(n)}{3} - 2\right).
\]

\begin{lem}
\label{lem:devphi}
For $n>1$ and $k\ge1$, we have
\[
\frac{1}{\varphi(n)^k}
\frac{\Phi_{n}^{(k)}(1)}{\Phi_{n}(1)}
=
k!\sum\limits_{(\ast)}
\prod_{i=1}^{k}\frac{(-1)^{i\lambda_i}}{\lambda_i!}
\left(\frac{B_i}{i!\cdot i}\right)^{\lambda_i}
\frac{J_i(n)^{\lambda_i}}{\varphi(n)^{k}}
+
O_{k}\left(\frac{n^{k-1}}{\varphi(n)^{k}}\right),
\]
where the summation $\sum\limits_{(\ast)}$ is as in Theorem~\ref{1overEuler}. 
\end{lem}
\begin{proof}
Since $J_h(n)\le n^h$ and $c(i,i)=1$,
it follows from \eqref{lehmer:power} that
\[
-s_i(n)
=
(-1)^{i}\frac{B_i}{i!}J_i(n)+O_{k}(n^{i-1}).
\]
Hence, by raising both
sides to the $\lambda_i$-th power, we have
\[
(-s_i(n))^{\lambda_i}
=
(-1)^{i\lambda_i}
\left(\frac{B_i}{i!}J_i(n)\right)^{\lambda_i}
+O_{k}(n^{i\lambda_i-1}).
\]
On substituting this estimate into \eqref{lehmer:newton},
the proof of the lemma is concluded
by taking the product over $1\le i\le k$
and noting that the error term is
$O_k(n^{\sum_{i=1}^{k}i\lambda_i-1})=O_k(n^{k-1})$
for each choice of $\lambda_1,\ldots,\lambda_k$
contributing to the sum $\sum\limits_{(\ast)}$.
\end{proof}

\begin{proof}[Proof of Theorem~\ref{1overEuler}]
By \eqref{1overEuler:trivial}, we may assume $k\ge2$.
By Lemma~\ref{lem:devphi} and Corollary~\ref{cor:philog},
\[
\sum_{1<n\le x}\frac{1}{\varphi(n)^k}\frac{\Phi_n^{(k)}(1)}{\Phi_n(1)}
=
k!\sum\limits_{(\ast)}
\prod_{i=1}^{k}\frac{(-1)^{i\lambda_i}}{\lambda_i!}
\left(\frac{B_i}{i!\cdot i}\right)^{\lambda_i}
\sum_{n\le x}J_{\bm{e}(\bm{\lambda})}(n)
+O_{k}(\log x),
\]
where we used the summation $\sum\limits_{(\ast)}$
and the indices $\bm{e}(\bm{\lambda})$
defined in Theorem~\ref{1overEuler}.
Note that every index $\bm{e}(\bm{\lambda})$ 
appearing on the right-hand side
has weight 
\[
w=\sum_{i=1}^{\infty}ie_i(\bm{\lambda})
=
\sum_{i=1}^{k}i\lambda_i-k
=
0.
\]
Trivially $|e_1(\bm{\lambda})|\le k$ and hence,
by applying Theorem~\ref{Thm:MV_Jordan} and using $k\ge2$, we obtain
\[
\sum_{1<n\le x}\frac{1}{\varphi(n)^k}\frac{\Phi_n^{(k)}(1)}{\Phi_n(1)}
=
\mathfrak{S}_k(\Phi)x
+
\sum_{r=1}^{k}C_{r}(\log x)^{r}
+
O_{k}((\log x)^{2k/3}(\log\log x)^{4k/3}),
\]
where
\[
\mathfrak{S}_k(\Phi)
:=
(-1)^kk!\sum\limits_{(\ast)}
\prod_{i=1}^{k}\frac{1}{\lambda_i!}
\left(\frac{B_i}{i!\cdot i}\right)^{\lambda_i}
\mathfrak{S}_{\bm{e}(\bm{\lambda})}.
\]
\end{proof}

\subsection{The second derivative of $\Phi_n$ at $-1$}
We prove an analogous result
for the normalized second derivative of $\Phi_n$ at $-1$.

\begin{Thm}
\label{Phi-1}
We have
$$\sum_{2 < n \leq x} \frac{1}{\varphi(n)^2} \frac{\Phi_n''(-1)}{\Phi_n(-1)}=\frac{x}{48}(5\mathfrak{S}_{(-2,1)}+12) + c_2\log^2x+O((\log x)^{4/3}(\log \log x)^{8/3}),$$
where $c_2$ 
is a constant and $\mathfrak{S}_{(-2,1)}$ computed via \eqref{MV_Jordan:singular_e} has the Euler product
\[
\mathfrak{S}_{(-2,1)} = \prod_p \Big(1+\frac{2}{p(p-1)}\Big).
\]
\end{Thm}
\begin{proof}
By~\cite[Corollary 22]{HM} it follows that for $n \ge 3$ we have
\[
\frac{\Phi_n''(-1)}{\Phi_n(-1)}
=
\frac{\varphi(n)}{4}\left(\varphi(n) + a_n\Psi(n) - 2\right),
\]
where
\[
a_n
=
\begin{cases}
1 & \hbox{if } n \hbox{ is odd},\\
1/9 & \hbox{if } 2\parallel n,\\
1/3 & \hbox{otherwise.}
\end{cases}
\]
Using the above and Lemma~\ref{cor:philog}, it now follows that
\begin{equation} 
\label{beginlem-2}
\sum_{2<n\leq x}
\frac{1}{\varphi(n)^2}\frac{\Phi''_n{(-1)}}{\Phi_n{(-1)}}
=
\frac{x}{4}
+\frac{1}{4}\sum_{n \leq x} a_n \frac{\Psi(n)}{\varphi(n)}
+O(\log x).
\end{equation}
Note that
\[
\sum_{n \leq x}\left(a_n-\frac{1}{3}\right) \frac{\Psi(n)}{\varphi(n)}
=
\frac{2}{3} \sum_{\bfrac{n \leq x}{2 \nmid n}}\frac{\Psi(n)}{\varphi(n)}
-\frac{2}{9} \sum_{\bfrac{n \leq x}{2\parallel n}}\frac{\Psi(n)}{\varphi(n)}
=
\frac{2}{3} \sum_{\bfrac{n \leq x}{2 \nmid n}}\frac{\Psi(n)}{\varphi(n)}
-\frac{2}{3} \sum_{\bfrac{n \leq x/2}{2 \nmid n}} \frac{\Psi(n)}{\varphi(n)},
\]
and so
\[
\sum_{n \leq x} a_n \frac{\Psi(n)}{\varphi(n)}
=
\frac{1}{3} \sum_{n \leq x} \frac{\Psi(n)}{\varphi(n)}
+\frac{2}{3} \sum_{\bfrac{n \leq x}{2 \nmid n}} \frac{\Psi(n)}{\varphi(n)}
-\frac{2}{3} \sum_{\bfrac{n \leq x/2}{2 \nmid n}} \frac{\Psi(n)}{\varphi(n)}.
\]
By Theorem~\ref{Thm:MV_Jordan}, for the first sum, we have
\[
\sum_{n \leq x} \frac{\Psi(n)}{\varphi(n)} = \mathfrak{S}_{(-2,1)}x+c'_1\log^2x+O((\log x)^{4/3}(\log \log x)^{8/3}).
\]
On noting that
\[
1_{2 \nmid n}\frac{\Psi(n)}{\varphi(n)} = \prod_{p|n} (1+\theta_p), \quad \theta_p = \frac{2}{p-1}\ (p\neq 2), \quad \theta_2=-1,
\]
we get on applying Lemma \ref{Lem:weight_totient_BP},
\[
\sum_{\bfrac{n \leq x}{2 \nmid n}} \frac{\Psi(n)}{\varphi(n)}
=
\frac{1}{4} \mathfrak{S}_{(-2,1)}x+c'_2\log^2x+O((\log x)^{4/3}(\log \log x)^{8/3}).
\]
Combining the results above we obtain
\[
\sum_{n \leq x} a_n \frac{\Psi(n)}{\varphi(n)} = \frac{5}{12} \mathfrak{S}_{(-2,1)}x+4c_2\log^2 x+O((\log x)^{4/3}(\log \log x)^{8/3}),\]
which, together with
\eqref{beginlem-2}, concludes the proof.
\end{proof}

\subsection{Schwarzian derivative of $\Phi_n$ at $1$} 
\label{mainproofs}
Given a holomorphic function $f$ of one complex variable $z$,
we define its \emph{Schwarzian derivative}, cf. \cite{ovsienko}, as 
\[
S(f(z))
=
\frac{f'''(z)}{f'(z)}
-\frac{3}{2}\left(\frac{f''(z)}{f'(z)}\right)^2.
\]

\begin{Thm}
\label{schwarziander}
We have
\[
\sum_{n\le x}\frac{S(\Phi_n(1))}{\varphi(n)^2}
=
-\frac{1}{24}({\mathfrak S}_{(-4,2)}+3)x
+c_4\log^4 x
+c_3\log^3 x
+O((\log x)^{8/3}(\log \log x)^{16/3}),
\]
where $c_3,c_4$ are constants
and ${\mathfrak S}_{(-4,2)}$ computed via \eqref{MV_Jordan:singular_e} has the Euler product
\[
{\mathfrak S}_{(-4,2)}=\prod_p \Big( 1+\frac{4}{(p-1)^2}\Big).
\]
\end{Thm}
\begin{proof}
By Lemma \ref{lem:lehmer}, we have for $n \ge 2$
\[S(\Phi_n(1)) = -\frac{\varphi(n)^2}{8} - \frac{\Psi(n)^2}{24} + \frac{1}{2},\]
and thus
\[
\sum_{n \leq x} \frac{S(\Phi_n(1))}{\varphi(n)^2}
=
-\frac{1}{8}\sum_{n \leq x}1
-\frac{1}{24}\sum_{n \leq x} \frac{\Psi(n)^2}{\varphi(n)^2}
+\frac{1}{2}\sum_{n \leq x} \frac{1}{\varphi(n)^2}.
\]
The last sum is bounded by a constant by Proposition \ref{Prop:MV_Jordan_toy}
with $\Je(n)=1/\varphi(n)^2$ and $\beta=0$. 
The result now follows on applying Theorem \ref{Thm:MV_Jordan}
with $\bm{e}=(-4,2)$. 
\end{proof}

\begin{remark*}
On applying the elementary Theorem~\ref{exconjJ},
we obtain Theorems \ref{Phi-1} and \ref{schwarziander}
with error terms $O(\log^2 x)$ and $O(\log^4 x)$, respectively.
\end{remark*}

\section*{Acknowledgement}
A large portion of this paper was written during the stay of the second and the fourth author at the Max Planck Institute for Mathematics (MPIM) in September 2018. They would like to thank Pieter Moree for inviting them and they gratefully acknowledge the support, hospitality as well as the excellent environment for collaboration at the MPIM. The second  author is supported by the Austrian Science Fund (FWF): Project F5505-N26 and Project F5507-N26, which are part of the special Research Program  ``Quasi Monte Carlo Methods: Theory and Application".
The fourth author is supported by Grant-in-Aid for JSPS Research Fellow (Grant Number: JP16J00906).


\bibliographystyle{elsarticle-num}

\end{document}